\pdfoutput = 1

\documentclass{amsart}
\usepackage{a4wide}
\usepackage{hyperref}
\usepackage{amsrefs}
\usepackage{amssymb}
\usepackage{graphicx,color}
\usepackage[legacycolonsymbols]{mathtools}

\usepackage{graphicx}
\graphicspath{{pictures/handwrittenfigure/}}

\theoremstyle{remark}
\newtheorem{remark}{Remark}

\theoremstyle{plain}
\newtheorem{theorem}{Theorem}
\newtheorem{corollary}{Corollary}
\newtheorem{lemma}[theorem]{Lemma}
\newtheorem{prop}[theorem]{Proposition}

\theoremstyle{definition}
\newtheorem{definition}{Definition}
\newtheorem{notation}{Notation}

\newcommand{\ind}{\operatorname{ind}}

\newcommand{\St}{\operatorname{St_{(2)}}}

\title{Arnold Strangeness of surface immersions}
\author{Noboru Ito}
\address{Department of Mathematics, Faculty of Engineering, Shinshu University, 4-17-1 Wakasato, Nagano 380-8553, Japan}
\email{nito@shishu-u.ac.jp}
\author{Hiroki Mizuno}
\address{Department of Science and Technology, Graduate School of Medicine, Science and Technology, Shinshu University, 3-1-1 Asahi, Matsumoto, Nagano 390-8626, Japan}
\email{22hs602j@shinshu-u.ac.jp}
\date{June 29, 2025}

\begin{document}
\begin{abstract}
It is known that for any smooth sphere eversion, the number of quadruple point jumps is always odd~\cite{NelsonBanchoff1981}.      
In this paper, we define an integer-valued function that detects and classifies jumps involving quadruple points and triple-line tangencies.      
Our function provides a higher-dimensional analogue of the Arnold strangeness invariant for plane curves. It classifies quadruple point jumps into the five geometrically distinct cases 
based on coorientation data and reflects finer geometric features for generic immersions of closed surfaces into $\mathbb{R}^3$.  
\end{abstract}
\maketitle

\section{Introduction}\label{Intro}
This paper is motivated by the question of how quadruple points arise during sphere eversions, following Smale's classical result~\cite{Smale1958} that the unit sphere in $\mathbb{R}^3$ is regular homotopic to its reflection.   In particular, for any smooth sphere eversion, the number of quadruple point jumps is always odd~\cite{NelsonBanchoff1981}.  
This fundamental result suggests the existence of a topological invariant that detects and quantifies the algebraic contributions of such  jumps.  
The integer values arise naturally when one aims to detect the five geometrically distinct cases of quadruple point jumps classified by Goryunov~\cite{Goryunov1997}.  In contrast to the Arnold strangeness invariant of plane curves, which changes by $\pm 1$ under each triple point crossing, our integer-valued function encodes finer geometric information since both the domain and the target dimensions are higher.   Moreover, the invariant allows one to formulate explicit lower bounds on the number of quadruple point jumps occurring during sphere eversions as in Section~\ref{sec:Appl}.    

We define an integer-valued invariant $\St$ for generic surface immersions $\Sigma \looparrowright \mathbb{R}^3$, which changes  under singularity jumps.  In the case of sphere immersions, $\St$ refines a higher-dimensional analogue of the Arnold strangeness invariant~  \cite{Arnold1994Book}, obtained by lifting both the domain and the target dimensions by one. 
  
Following Goryunov~\cite{Goryunov1997}, we classify quadruple point jumps~\textup{(Q)} into the five geometrically distinct cases, denoted by $Q^4$, $Q^3$, and $Q^2$.  We coorient $Q^4$ and $Q^3$ in the direction of increasing the number of outward-pointing faces of the local tetrahedron, while there is no local way to distinguish between the two sides of the $Q^2$. We begin with our main result. 
\begin{theorem}\label{thmSurface}
Let $\Sigma$ be a closed oriented surface and let $S$ be a generic surface immersion $S : \Sigma \looparrowright \mathbb{R}^3$; $T(S)$ the set of triple points of $S$; and $\ind(t)$ the Alexander numbering of the triple point $t$, defined as the average over the indices of the eight regions adjacent to $t$. 
Then the value
\[\St(S)= \sum_{t \in T(S)} \ind(t)
\]
behaves under quadruple point jumps  \textup{(Q)} as in Table~\ref{table:MainDiff}.  
\begin{table}[h!]
\caption{Changes in $\St$ under each case of quadruple point jumps~\textup{(Q)}.}\label{table:MainDiff}
{\rm
\begin{tabular}{|c|c|c|}\hline
Jump cases & Coorientations of local tetrahedra (\# outward sheets) & Changes in $\St$ \\ \hline
Negative $Q^4$ & $4\to0$ & $+4$ \\ \hline
Negative $Q^3$ & $3\to1$ & $+2$ \\ \hline
 $Q^2$ & $2\to2$ & $0$ \\ \hline
Positive $Q^3$ & $1\to3$ & $-2$ \\ \hline
Positive $Q^4$ & $0\to4$ & $-4$ \\ \hline
\end{tabular}
}
\end{table}
For \textup{(T)}, the change is an integer $2n$ that is $2 \ind(t)$ of the   triple points   created during a jump \textup{(T)}.
Moreover, $\St$ remains unchanged under the other types of singularity jumps:  \textup{(E)} and \textup{(H)}.    
\end{theorem}


Since it is quite natural to consider mathematical tools, i.e. topological invariants, to detect and classify how quadruple point jumps, classified into five cases, and triple-line tangency jumps (\textup{T}) occur during the eversion process, we provide its application in  Section~\ref{sec:Appl}.

The plan of this paper is as follows.  
In Section~\ref{sec:Def}, we prepare definitions to define our invariant $\St$, following Goryunov~\cite{Goryunov1997}.   
In Section~\ref{sec:Proof}, we prove the main result (Theorem~\ref{thmSurface}). In Section~\ref{sec:Appl}, we present  applications of our main  result.

\section{Definitions}\label{sec:Def}
\begin{definition}
A smooth immersion from a closed oriented surface into $\mathbb{R}^3$ is said to be \emph{generic} if its image locally exhibits only transverse intersections of two or three smooth sheets.
\end{definition}

As long as there is no risk of confusion, we simply identify a generic immersion with its image. 

Let $\mathcal{M}:=C^\infty (\Sigma, \mathbb{R}^3)$ denote the space of smooth maps from a closed oriented surface $\Sigma$ into $\mathbb{R}^3$; let $\mathcal{F} \subset \mathcal{M}$ be the subspace of smooth immersions; and let $\Omega \subset \mathcal{F}$ denote the open dense subspace of generic immersions.  The complement  $\Delta=\mathcal{F} \setminus \Omega$ is called the \emph{discriminant}.  A path $\gamma : [0, 1] \to \mathcal{F}$ from $\gamma(0)$ to $\gamma(1)$ is said to be \emph{generic} if $\gamma$ intersects $\Delta$ transversely at finitely many points.   According to Goryunov~\cite{Goryunov1997}, each such intersection point corresponds to a singular map at which one of the four codimension-one jumps, namely \textup{(E)}, \textup{(H)}, \textup{(T)}, or  \textup{(Q)} occurs: 

\begin{itemize}
  \item[\textup{(E)}] elliptic tangency of two sheets (Figure~\ref{fig:Emove});
  \item[\textup{(H)}] hyperbolic tangency of two sheets (Figure~\ref{fig:Hmove});
\item[\textup{(T)}] (for `triple'), tangency of the line of intersection of two sheets to another sheet (Figure~\ref{fig:Tmove});
\item[\textup{(Q)}] (for `quadruple'), four sheets intersecting at the same point (Figure~\ref{fig:Qmove}).
\end{itemize} 
 \begin{figure}[htbp] 
    \includegraphics[width=10cm]{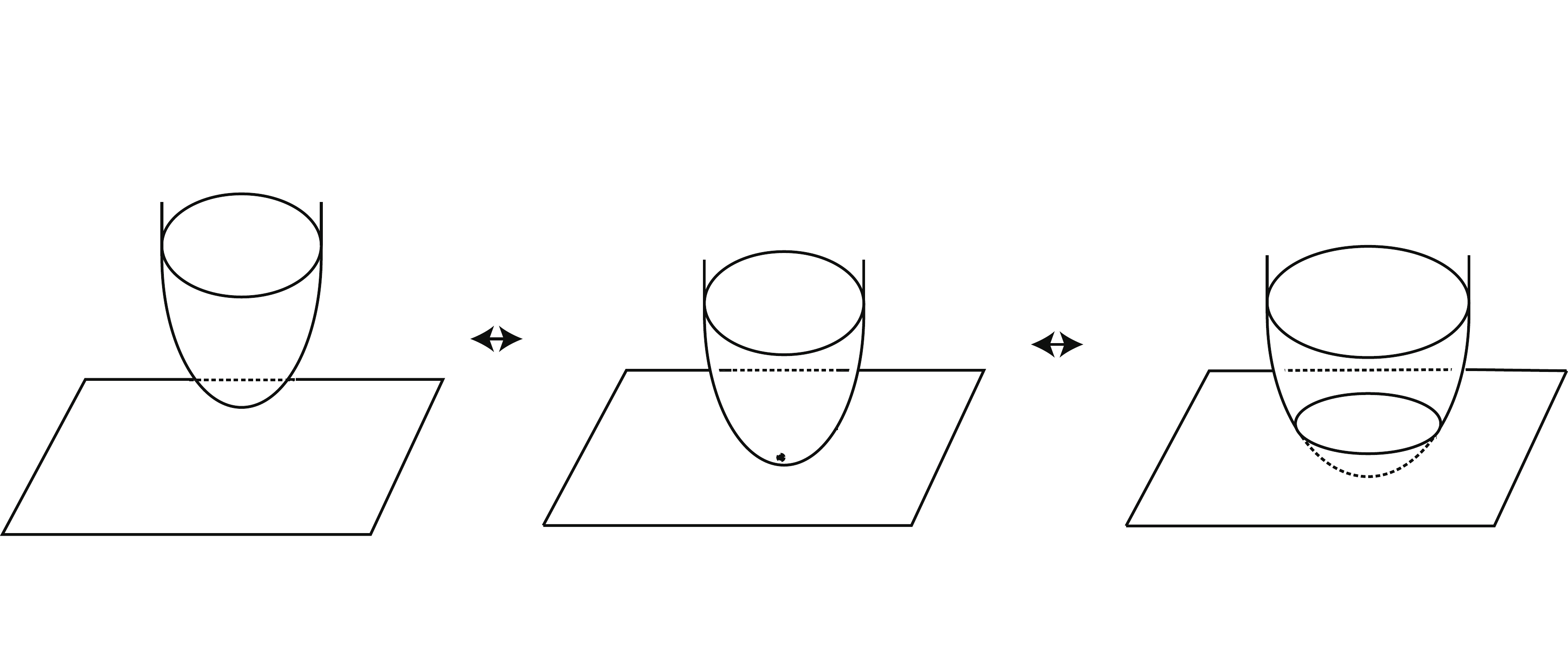} 
\vspace{-5mm}    \caption{\textup{(E)} : Elliptic tangency of two sheets.}
    \label{fig:Emove}
 \end{figure}
  \begin{figure}[htbp] 
     \includegraphics[width=10cm]{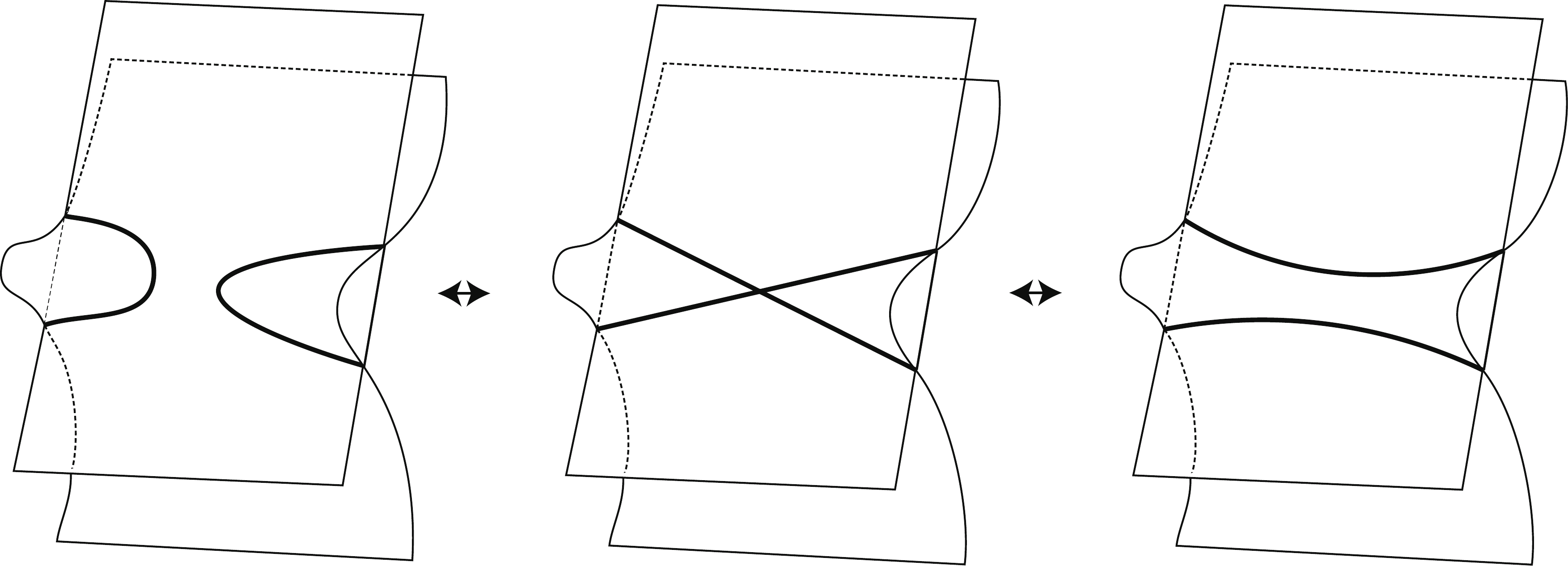} 
     \caption{\textup{(H)} : Hyperbolic tangency of two sheets.}
     \label{fig:Hmove}
  \end{figure}
  \begin{figure}[htbp] 
     \includegraphics[width=10cm]{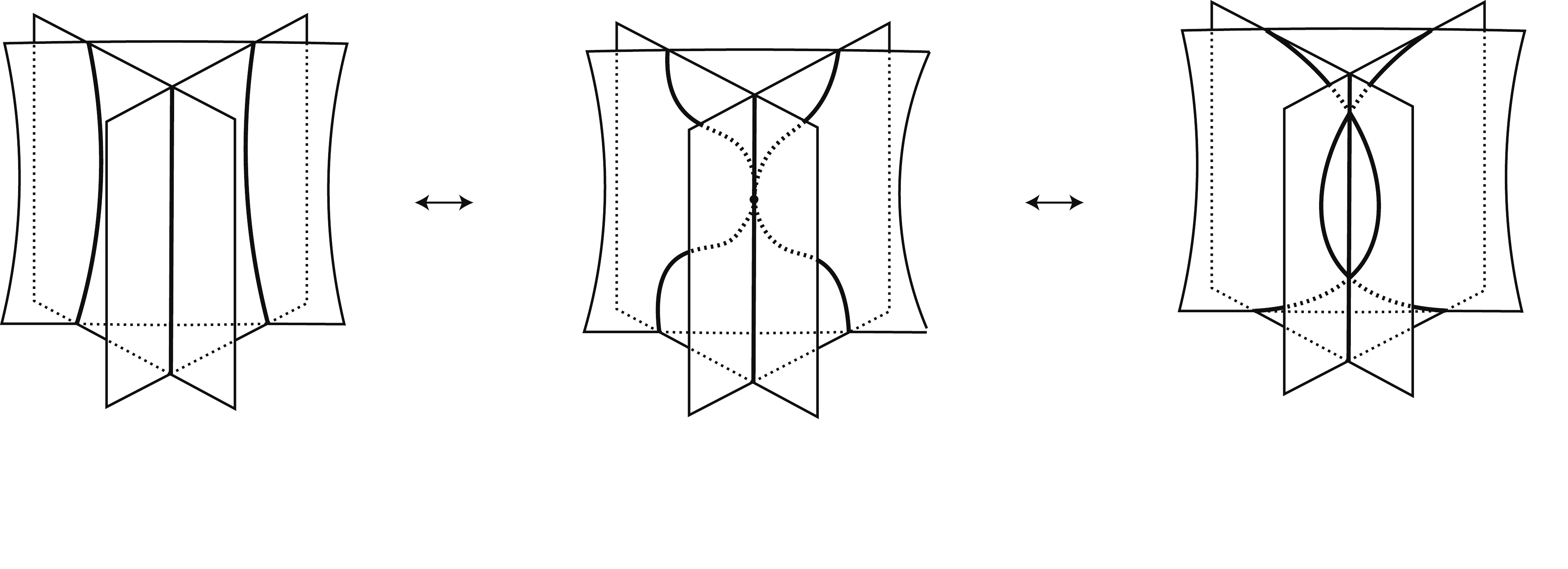} 
     \vspace{-4mm}    
     \caption{\textup{(T)} : Tangency of the line of intersection of two sheets to another sheet.}
     \label{fig:Tmove}
  \end{figure}
  \begin{figure}[htbp] 
     \includegraphics[width=10cm]{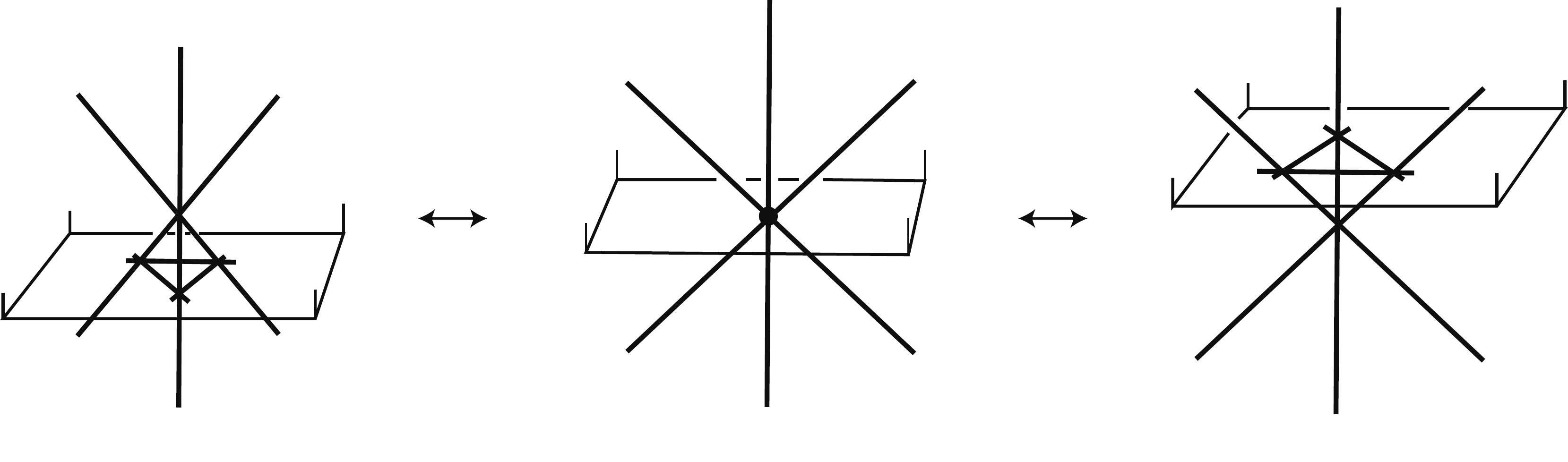} 
     \caption{\textup{(Q)} : A sheet crossing over a triple point formed by three other sheets. For the quadruple point, four sheets intersect at the same point.}
     \label{fig:Qmove}
  \end{figure}  
  
\begin{definition}
Two generic immersions $f$ and $g$ in $\Omega$ are said to be \emph{generically  regularly homotopic} if there is a generic path in $\mathcal{F}$ connecting them.  Equivalently, $f$ and $g$ are connected by a \emph{generic regular homotopy} if they can be transformed by a finite sequence of diffeomorphisms and the four local jumps \textup{(E)}, \textup{(H)}, \textup{(T)}, and \textup{(Q)}.   
\end{definition}
Hence, we study generic immersions under generic regular homotopy.  The formulation we adopt here reflects our focus on Smale's sphere eversion.  In contrast to the paper ~\cite{Goryunov1997}, which includes additional local moves involving pinch points, we restrict our attention to those singularities encountered in generic one-parameter families of immersions.  That is, our category of interest is generated entirely by generic immersions and generic regular homotopies arising from the four codimension-one jumps and diffeomorphisms.   

Given a generic immersion $S : \Sigma \looparrowright \mathbb{R}^3$, $\mathbb{R}^3$ is decomposed into the image $S(\Sigma)$ and its complement.  Moreover, the image of a generic immersion $S$ consists of \emph{sheets} (locally smooth pieces), \emph{double lines} (pairwise intersections of sheets), and \emph{triple points} (triple intersections of sheets).    
A connected  component of the complement of the image of a generic immersion is called a  \emph{region}.    
The coorientation of a generic immersion 
 is induced by the orientation of the domain surface. 
For a generic immersion, we assign numbers to the set of regions  based on its coorientation (Figure~\ref{fig:coori}).   
\begin{figure}[htbp] 
\includegraphics[width=5cm]{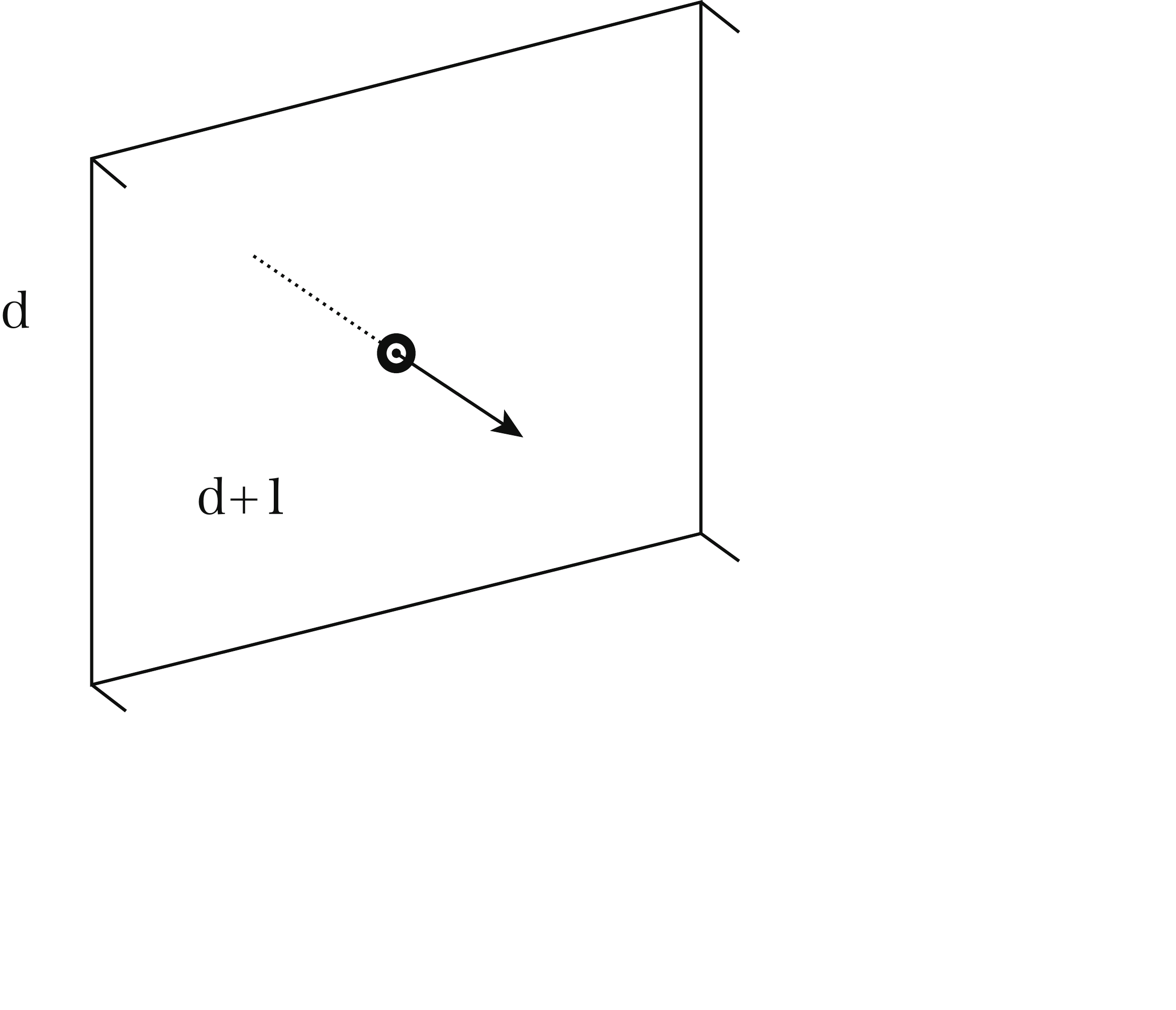}
    \vspace{-1cm}
   \caption{A coorientation of a sheet and indices.  The direction of the coorientation of the sheet is indicated by short segments.}
   \label{fig:coori}
\end{figure}
More precisely, the assignment proceeds as follows.  First, we assign $-\frac{3}{2}$ to the region containing the point at $\infty$.  Then, we increase or decrease the number along the coorientation.  In this way, we assign numbers to the complement of the image of generic immersion, and this defines the map $\ind : \{ \textrm{regions} \} \to \mathbb{Z}-\frac{1}{2}$, called the \emph{Alexander numbering}.     
\begin{definition}\label{def:opposite}
Any triple point $t$ is formed by three sheets (see Figure~\ref{fig:8-regions}).  Then two regions adjacent to a triple point $t$ are said to be \emph{opposite} if they are separated by exactly three sheets. 

For example, when we consider the triple point at the origin in $\mathbb{R}^3$ formed by the three coordinate planes $x=0$, $y=0$, and $z=0$, the regions containing  $(-1,-1,-1)$ and $(1,1,1)$ are opposite to each other.
\end{definition}
\begin{definition}\label{def:ind}
Let $\Sigma$ be a closed oriented surface,  
let $S$ be a generic immersion $S \colon \Sigma \looparrowright \mathbb{R}^3$, and let $T(S)$ be the set of triple points of $S$.   
We define a function $T(S) \to  \mathbb{Z};$ $t \mapsto \ind(t)$   
by averaging the values of $\ind$ over the eight regions that are locally adjacent to~$t$ (see Figure~\ref{fig:8-regions}, Lemma~\ref{lem:ind}).   The value $\ind(t)$ is called the \emph{index} of a triple point $t$.       
\end{definition}
The definition of $\ind(t)$ will be used in the definition of the invariant $\St$.  
We note that this definition is equivalent to defining $\ind(t)$ as the average of the values of $\ind$ over a pair of opposite regions adjacent to the triple point $t$; although there are four such choices, they all give the same result.  This fact will be useful when computing the change in $\St$ under the jump \textup{(Q)}. 

While the function $\St$ is integer-valued, the use of a half-integer valued Alexander numbering, defined as the average across regions, provides a convenient framework for the computation in Section~\ref{sec:Proof}.  

\section{Proofs}\label{sec:Proof}
In order to prove Theorem~\ref{thmSurface},  
we consider  Propositions~\ref{Prop:Tmove} and \ref{Prop:Qmove}, and Lemmas~\ref{lem:ind} and \ref{lem:oprg}, which  concern the jumps \textup{(T)} and \textup{(Q)}---singularity jumps involving triple points.   
\begin{prop}\label{Prop:Tmove}
    Let $t$ be a triple point that appears or disappears in a jump \textup{(T)}.  
    Then $\St$ changes by $2\ind(t)$ under the jump \textup{(T)}. 
  \end{prop}
  \begin{proof}
    To prove this proposition, it suffices to compare the difference between $\St(S_0)$ before and $\St(S_1)$ after the jump \textup{(T)}.  The two triple points $t$ and $t'$ that appear in the jump \textup{(T)} have the same value of $\ind$   
    (see Figure~\ref{fig:tript_Tmove}).  
\begin{figure}[htbp] 
       \includegraphics[width=8cm]{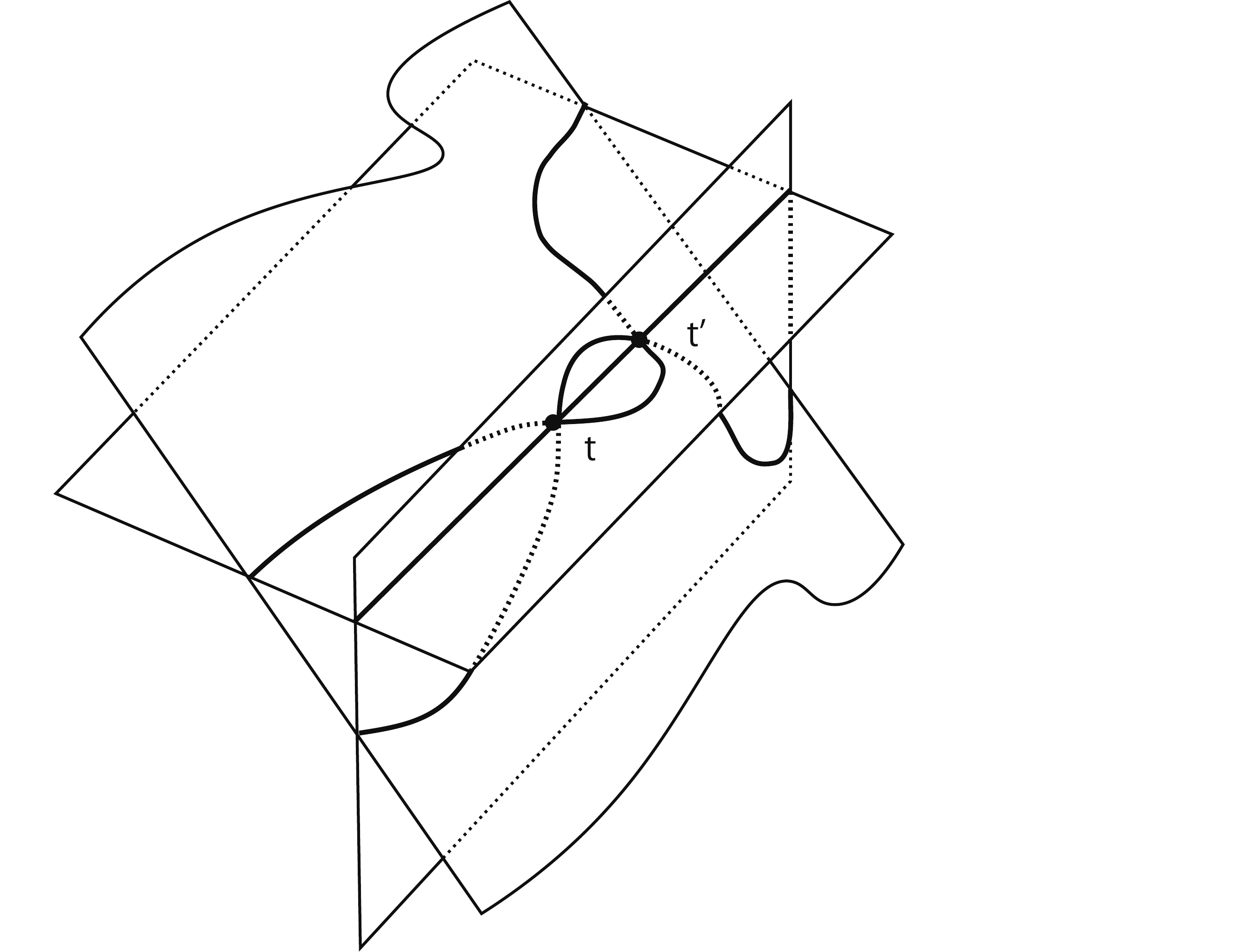} 
       \caption{Two triple points $t$ and $t'$ that appear in a single jump \textup{(T)}.}
       \label{fig:tript_Tmove}
    \end{figure}
In other words, the following holds 
    \begin{align*}
     \St(S_1) - \St(S_0) &= \left(\ind(t) + \ind(t')+\St(S_0)\right)-\St(S_0) \\
     &= \ind(t) + \ind(t') \\ 
     &= \ind(t) + \ind(t) \\
   &=2\ind(t).    
\end{align*}
  \end{proof} 
  \begin{lemma}\label{lem:ind}
     Let $d$ be the minimum of the indices of the eight regions adjacent to a triple point $t$.  Then \[\ind(t) = d + \frac{3}{2}, \]
and in particular,  
\[
\ind(t) \in \mathbb{Z}.  
\]
  \end{lemma}
\begin{proof}
By using the coorientation of the three sheets that form the triple point $t$, the indices of the eight regions adjacent to $t$ are automatically determined.  Therefore, the indices of regions adjacent to $t$ are $d, d+1, d+1, d+1, d+2, d+2, d+2,$ and $d+3$ (see Figure~\ref{fig:8-regions}).  
\begin{figure}[htbp] 
   \includegraphics[width=8cm]{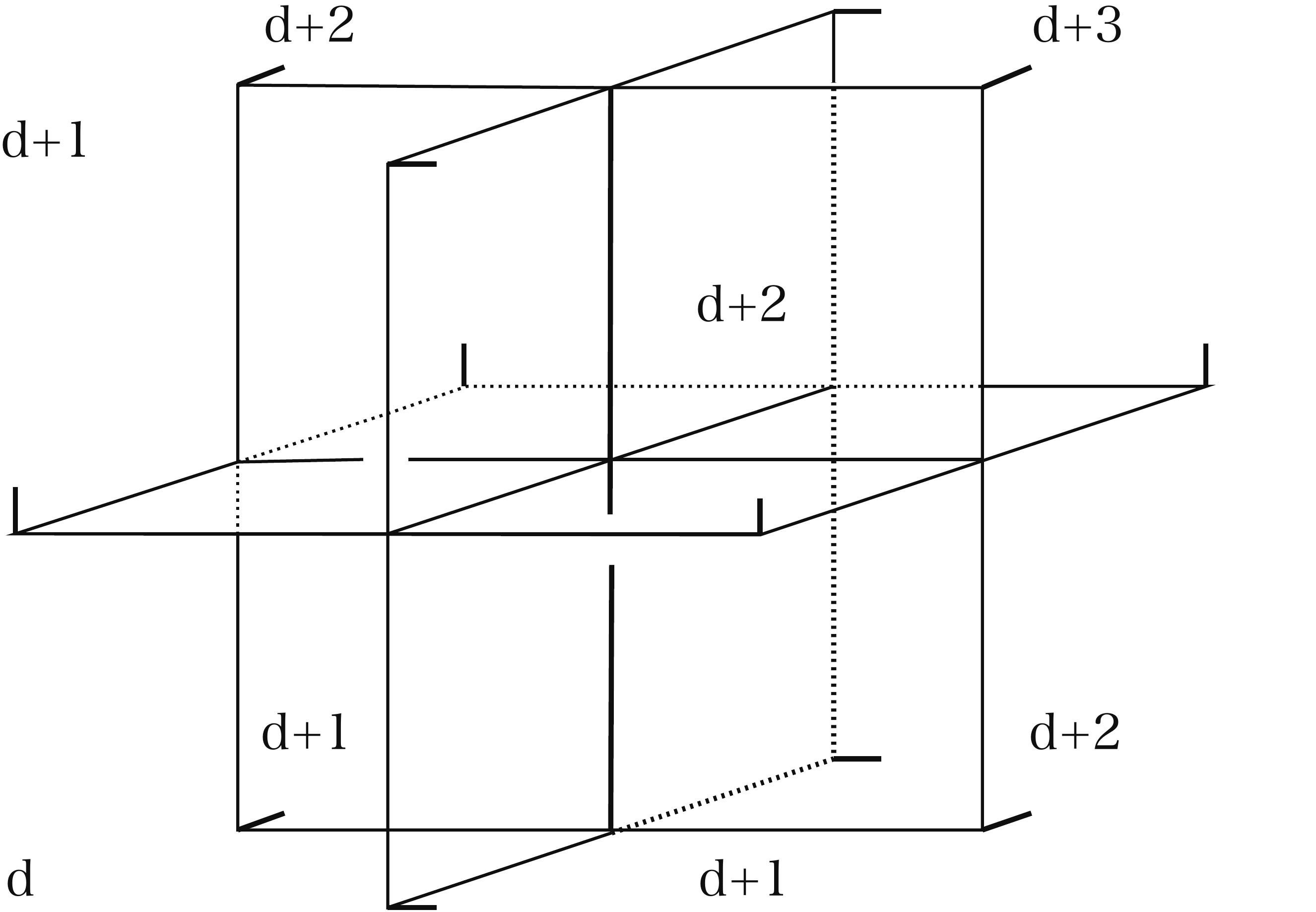} 
   \caption{The indices of the regions adjacent to $t$ are $d, d+1, d+1, d+1, d+2, d+2, d+2,$ and $d+3$.}
   \label{fig:8-regions}
\end{figure}
    Thus, 
    \[
    \ind(t) = \frac{1}{8} \left(d + d+1 + d+1 + d+1 + d+2 + d+2 + d+2 + d+3 \right) = d + \frac{3}{2}.
    \]
To show this is always an integer, recall  that we assign $-\frac{3}{2}$ to the region containing the point at $\infty$.  
Since $d \in \mathbb{Z}-\frac{1}{2}$, the value $\ind(t)=d+\frac{3}{2}$ lies in $\mathbb{Z}$. 
  \end{proof}
Before discussing the change of the values of $\St$ under any  jump \textup{(Q)}, let us mention Lemma~\ref{lem:oprg} that will help in calculating the difference in $\St$. 
It is technically convenient when  analyzing the changes under the jump \textup{(Q)}, to express the value $\ind(t)$ in terms of the indices of the opposite regions $r$ and $\hat{r}$  rather than using the minimal value.  
\begin{lemma}\label{lem:oprg}
The difference in $\ind$ values between two opposite regions adjacent to a triple point is determined by the coorientation of the three sheets that form the triple point.
\end{lemma}
\begin{proof}
Suppose that the indices of the regions adjacent to  $t$ are $d, d+1, d+1, d+1, d+2, d+2, d+2,$ and $d+3$.  Then the path from $r$ to $\hat{r}$ is classified into four cases,  according to the pair $\ind(r) \to \ind(\hat{r})$ as: 
\begin{itemize}
\item $d \to d+3$ (one path), 
\item $d+1 \to d+2$ (three paths), 
\item $d+2 \to d+1$ (three paths), 
\item $d+3 \to d$ (one path).  
\end{itemize}
There are exactly eight such paths, since there are eight regions adjacent to a given triple point $t$.  
This observation implies the claim.  
\end{proof}
\begin{notation}\label{Not:Qmove}
By definition, any jump \textup{(Q)} involves four sheets, among which three form a triple point, and the fourth sheet passes through the triple point.  By  Lemma~\ref{lem:oprg}, we introduce a refined notation to specify the four types of jump \textup{(Q)},  based on the direction in which the moving sheet $A$ passes from a region $r$ to its opposite region $\hat{r}$.  More concretely, 
\begin{itemize}
\item when $d \to d+3$, we call it a $\textup{Q}_3$-move, 
\item when $d+1 \to d+2$, we call it a $\textup{Q}_2$-move,
\item when $d+2 \to d+1$, we call it a $\textup{Q}_1$-move,
\item when $d+3 \to d$, we call it a $\textup{Q}_0$-move.  
\end{itemize}
\end{notation}
\begin{remark}
For the reader's convenience, we provide a table showing the  correspondence between our $\textup{Q}_i$-moves ($i=0,1,2,3$), defined above, and the classical Goryunov $Q^j$ strata ($j=2,3,4$) defined in terms of the coorientations of sheets. Here, the notation ``opposite $\textup{Q}_i$'' denotes the operation that reverses the direction of the $\textup{Q}_i$-move(Table~\ref{table:Comp}).  

\begin{table}[h!]
\caption{Comparison of Goryunov's notation and ours for strata.}\label{table:Comp}
\begin{tabular}{|c|c|c|}\hline
$Q^j$ strata &   $\textup{Q}_i$-moves & Opposite $\textup{Q}_i$-moves \\ \hline
Negative $Q^4$ stratum & & Opposite   $\textup{Q}_0$-move  \\ \hline
Negative $Q^3$ stratum & $\textup{Q}_3$-move & Opposite $\textup{Q}_1$-move \\ \hline 
 $Q^2$ stratum & $\textup{Q}_2$-move  & Opposite $\textup{Q}_2$-move \\ \hline 
Positive $Q^3$ stratum & $\textup{Q}_1$-move &  Opposite  $\textup{Q}_3$-move \\ \hline 
Positive $Q^4$ stratum & $\textup{Q}_0$-move & \\ \hline
\end{tabular}
\end{table}
\end{remark}
\begin{prop}\label{Prop:Qmove}
Let $S_0$ and $S_1$ be two generic immersions such that $S_1$ is obtained from $S_0$ by a \textup{Q}$_i$-move, where \textup{Q}$_i$-move is defined in ~Notation~\ref{Not:Qmove} $(i=0, 1, 2, 3)$.  
  The function $\St$ changes  under the jumps \textup{(Q)}. More precisely, 
  \[
\St (S_1) -\St(S_0) =
\begin{cases}
2 & \text{under \textup{Q}$_{3}$-move}  \\
0 & \text{under \textup{Q}$_{2}$-move}  \\
-2 & \text{under \textup{Q}$_{1}$-move}  \\
-4 & \text{under \textup{Q}$_{0}$-move} .
\end{cases}
\]
\end{prop}
\begin{proof}
Let $X$, $Y$, and $Z$ denote the three sheets that form the triple point appearing in the Q$_i$-move ($i=0, 1, 2, 3$), and let $A$ and $A'$ denote the moving sheets.  Let $t_{X,Y,Z}$ be the triple point formed by $\{X, Y, Z\}$.   Similarly, define $t_{A,X,Y}$, $t_{A,Y,Z}$, and $t_{A,Z,X}$ as the triple points formed by $\{A,X,Y\}$, $\{A,Y,Z\}$, and $\{A,Z,X\}$, respectively.
Fixing the coorientation of $A$ does not affect the generality of the argument.
Let $d$ denote the $\ind$ value of the tetrahedron appearing in the figure before the Q$_i$-move.
  The pattern of change in $\St$ depends on the coorientations of the sheets $X$, $Y$, and $Z$.  Hence, 
  we consider four possible cases, described in Notation~\ref{Not:Qmove}.  
Although $t_{X,Y,Z}$ refers to the same point before and after the move, we denote the triple point after the move by $t'_{X,Y,Z}$ to distinguish the value of $\ind$.

Let us first consider the Q$_3$-move. By the definition of coorientations of sheets, the distribution of $\ind$ values over the regions before and after the Q$_3$-move is illustrated in Figure~\ref{fig:DiffQ3}.  
\begin{figure}[htbp] 
    \vspace{-3cm}
 \includegraphics[width=16cm]{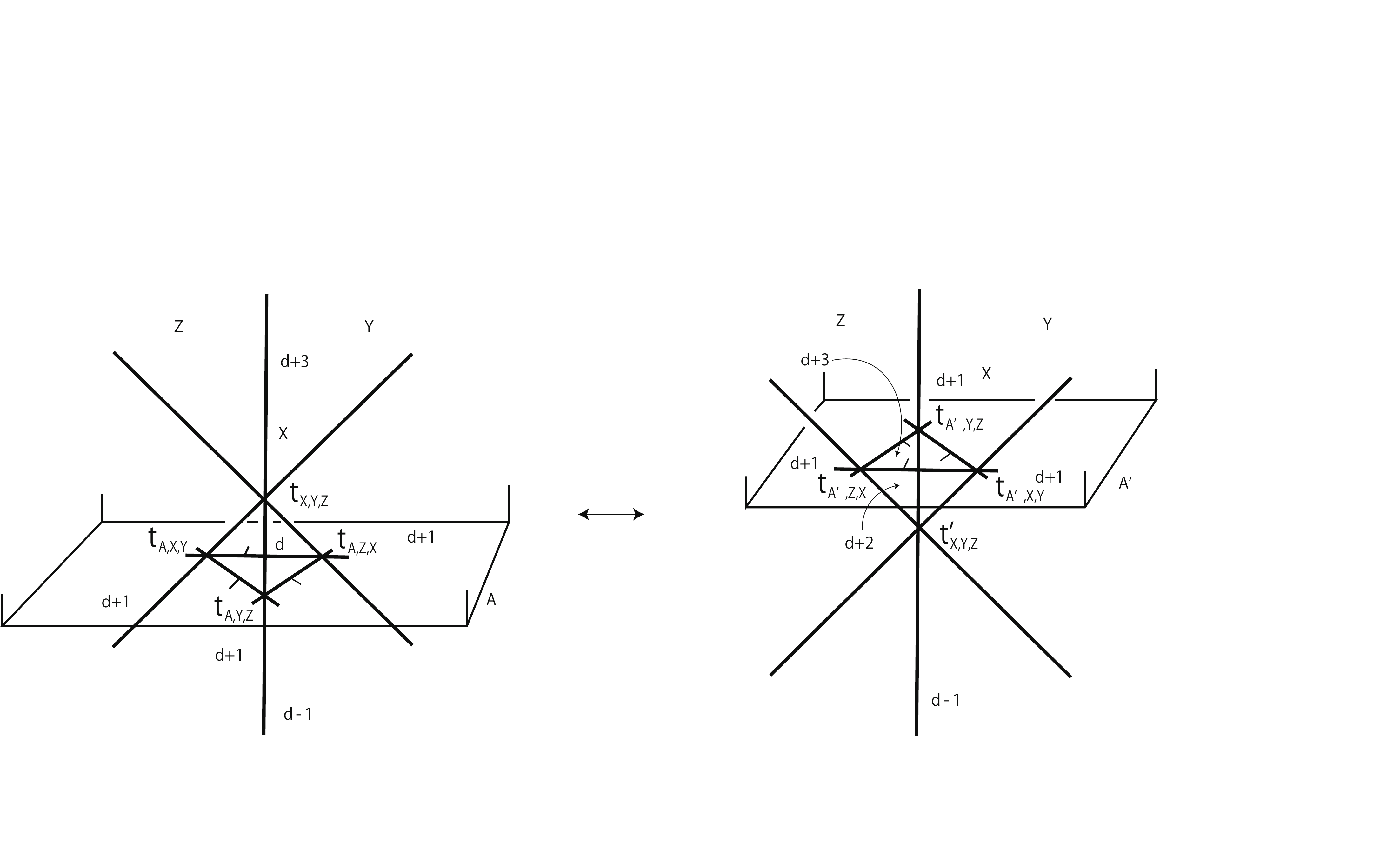}
 \vspace{-1cm}
 \caption{The indices for a single Q$_3$-move.}
   \label{fig:DiffQ3}
\end{figure}
In particular, the indices take the following values:
\[
  \begin{alignedat}{8}
    \ind(t_{A',X,Y}) &= d + \frac{3}{2} ,&\quad
      \ind(t_{A',Y,Z}) &= d + \frac{3}{2} ,&\quad
      \ind(t_{A',Z,X}) &= d + \frac{3}{2} ,&\quad
      \ind(t'_{X,Y,Z}) &= d + \frac{1}{2},\\
    \ind(t_{A,X,Y})  &= d + \frac{1}{2} ,&\quad
      \ind(t_{A,Y,Z})  &= d + \frac{1}{2} ,&\quad
      \ind(t_{A,Z,X})  &= d + \frac{1}{2}, &\quad
      \ind(t_{X,Y,Z}) &= d+ \frac{3}{2}.
  \end{alignedat}
\]
  The variation of $\St$ across the Q$_3$-move is as follows:
\begin{align*}
  \St(S_1) - \St(S_0)
  &= \ind(t_{A',X,Y}) + \ind(t_{A',Y,Z}) + \ind(t_{A',Z,X}) + \ind(t'_{X,Y,Z}) \\
  &\quad - \ind(t_{A,X,Y}) - \ind(t_{A,Y,Z}) - \ind(t_{A,Z,X}) - \ind(t_{X,Y,Z}) \\
  &=(d + \tfrac{3}{2}) + (d + \tfrac{3}{2}) + (d + \tfrac{3}{2}) + (d + \tfrac{1}{2}) \\
  &\quad - (d + \tfrac{1}{2}) - (d + \tfrac{1}{2}) - (d + \tfrac{1}{2}) - (d + \tfrac{3}{2}) \\
  &= 2.
\end{align*}

Next, we consider the second case, namely the Q$_2$-move. 
This case is obtained from the Q$_3$-move by reversing the coorientation of one of the sheets $\{X,Y,Z \}$. Due to the symmetry around the triple point $t_{X,Y,Z}$, the choice of which sheet to reverse does not affect the generality of the argument.  Let us assume that the coorientation of the sheet $X$ is reversed.  In this case, the values of $\ind$ assigned to the regions differ from those in the case of the Q$_3$-move. Using  Lemma~\ref{lem:ind}, we determine the distribution of $\ind$ values over the regions, as shown in Figure~\ref{fig:DiffQ2}.   
\begin{figure}
\vspace{-3cm}
\includegraphics[width=16cm]{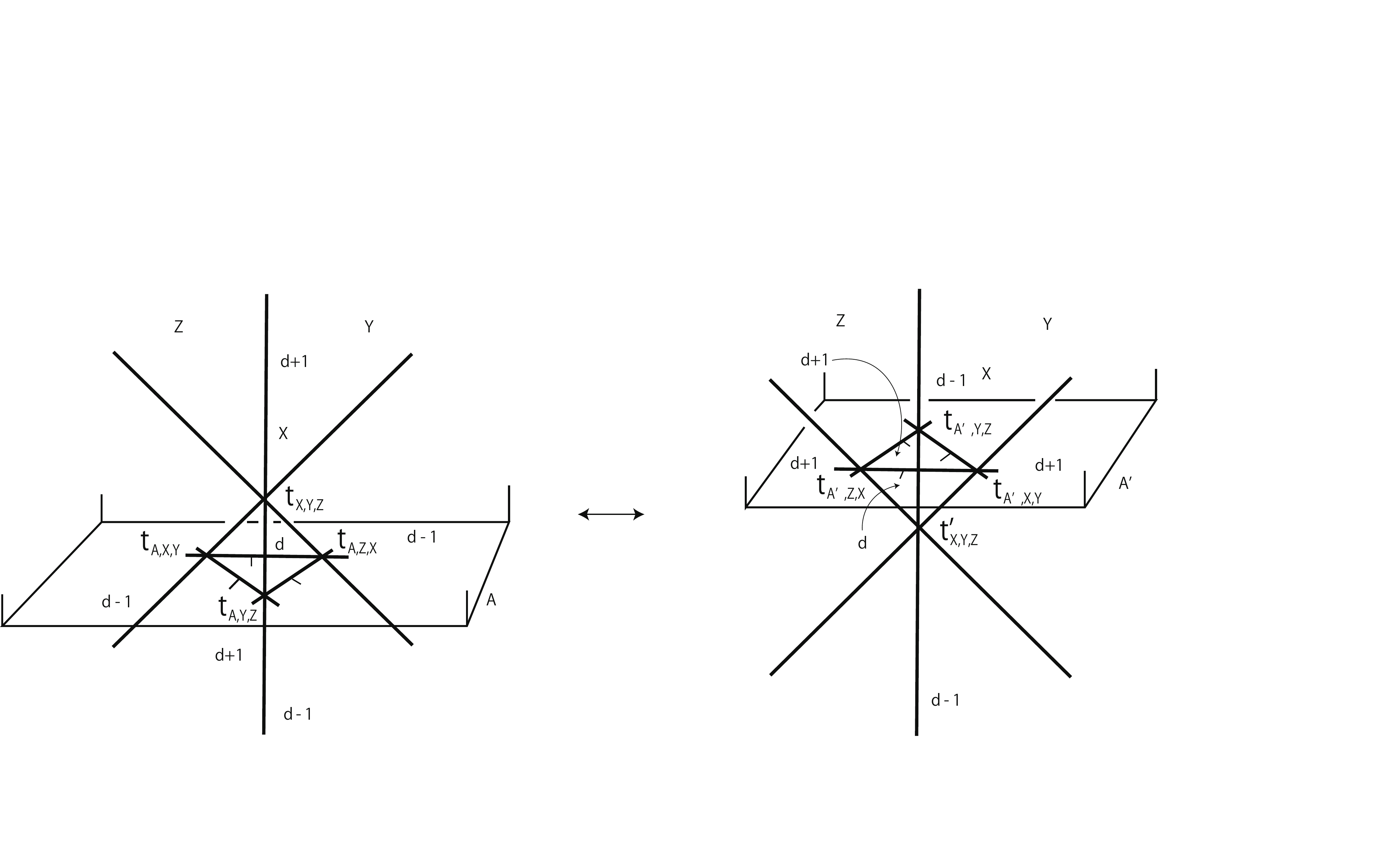}
\vspace{-1cm}
\caption{Indices for a single Q$_2$-move.}\label{fig:DiffQ2}
\end{figure}
The values of $\ind$ of triple points in this distribution are as follows:
 \[
  \begin{alignedat}{8}
    \ind(t_{A',X,Y}) &= d + \frac{1}{2} ,&\quad
      \ind(t_{A',Y,Z}) &= d - \frac{1}{2} ,&\quad
      \ind(t_{A',Z,X}) &= d + \frac{1}{2} ,&\quad
      \ind(t'_{X,Y,Z}) &= d - \frac{1}{2},\\
    \ind(t_{A,X,Y})  &= d - \frac{1}{2} ,&\quad
      \ind(t_{A,Y,Z})  &= d + \frac{1}{2} ,&\quad
      \ind(t_{A,Z,X})  &= d - \frac{1}{2} ,&\quad
      \ind(t_{X,Y,Z}) &= d + \frac{1}{2}.
  \end{alignedat}
\]
  Therefore, we have the following:
\begin{align*}
  \St(S_1) - \St(S_0)
  &= \ind(t_{A',X,Y}) + \ind(t_{A',Y,Z}) + \ind(t_{A',Z,X}) + \ind(t'_{X,Y,Z}) \\
  &\quad - \ind(t_{A,X,Y}) - \ind(t_{A,Y,Z}) - \ind(t_{A,Z,X}) - \ind(t_{X,Y,Z}) \\
  &= (d + \tfrac{1}{2}) + (d - \tfrac{1}{2}) + (d + \tfrac{1}{2}) + (d - \tfrac{1}{2}) \\
  &\quad - (d - \tfrac{1}{2}) - (d + \tfrac{1}{2}) - (d - \tfrac{1}{2}) - (d + \tfrac{1}{2}) \\
  &= 0.
\end{align*}

The Q$_1$-move is the opposite direction of the Q$_3$-move.
In this case, the change in the value of $\St$ simply follows the reverse of the Q$_3$ case.


  The case of the Q$_0$-move is obtained from the Q$_3$-move by reversing coorientations of $X$,$Y$, and $Z$.  
The distribution of the values of $\ind$ over regions in this case is shown in Figure~\ref{fig:DiffQ0}.  
\begin{figure}
    \vspace{-3cm}
    \includegraphics[width=16cm]{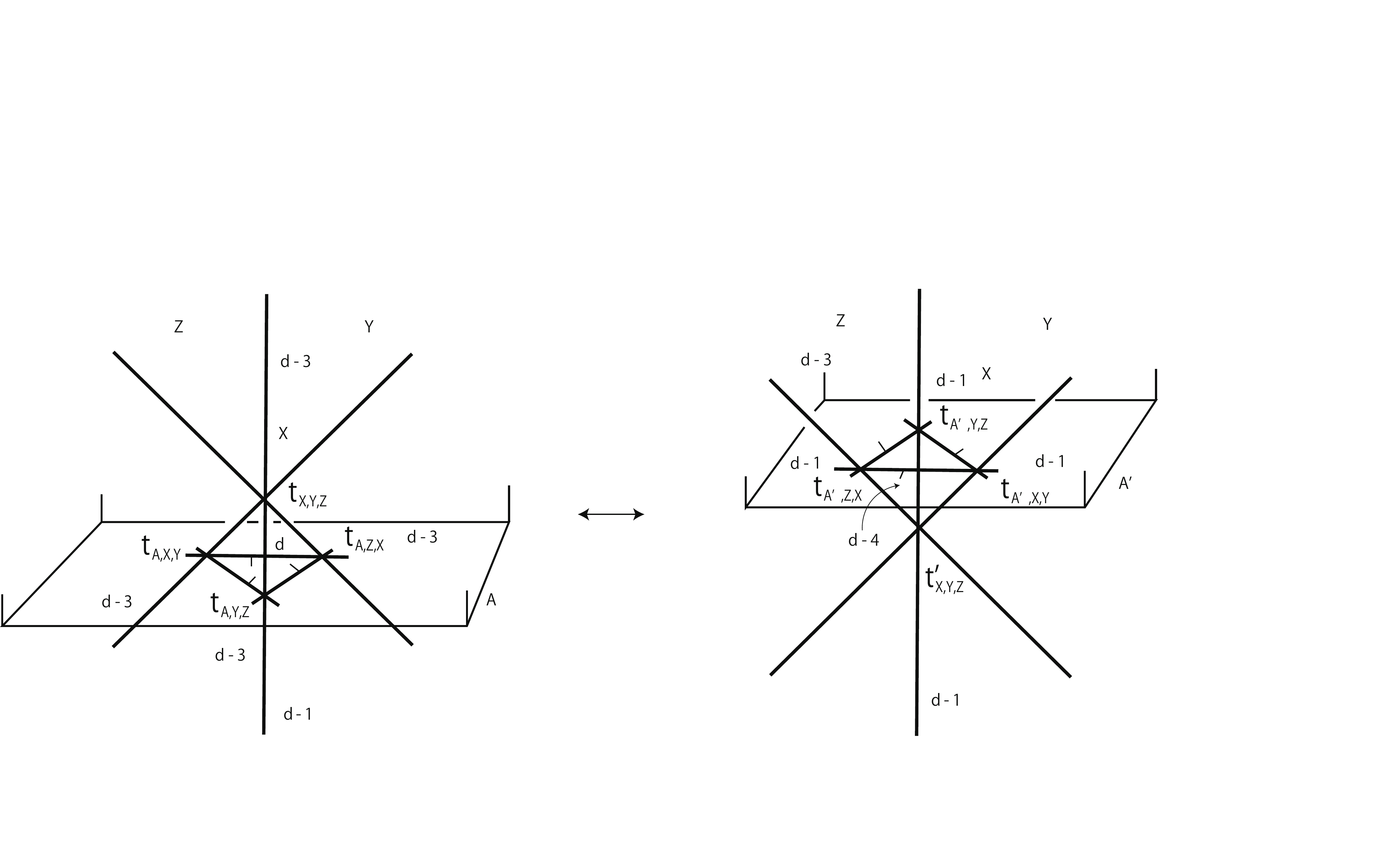}
    \vspace{-1cm}
    \caption{The indices for a single Q$_0$-move.}\label{fig:DiffQ0}
\end{figure}
As in the previous cases, we compute the $\ind$ of the triple points as follows:
 \[
  \begin{alignedat}{6}
    \ind(t_{A',X,Y}) &= d - \frac{5}{2} ,&\quad
      \ind(t_{A',Y,Z}) &= d - \frac{5}{2} ,&\quad
      \ind(t_{A',Z,X}) &= d - \frac{5}{2} ,&\quad
      \ind(t'_{X,Y,Z}) &= d - \frac{5}{2},\\
    \ind(t_{A,X,Y})  &= d - \frac{3}{2} ,&\quad
      \ind(t_{A,Y,Z})  &= d - \frac{3}{2} ,&\quad
      \ind(t_{A,Z,X})  &= d - \frac{3}{2}&\quad
      \ind(t_{X,Y,Z}) &= d - \frac{3}{2}.
  \end{alignedat}
\]
  Consequently, we obtain: 
\begin{align*}
  \St(S_1) - \St(S_0)
  &=  \ind(t_{A',X,Y}) + \ind(t_{A',Y,Z}) + \ind(t_{A',Z,X}) + \ind(t'_{X,Y,Z}) \\
  &\quad - \ind(t_{A,X,Y}) - \ind(t_{A,Y,Z}) - \ind(t_{A,Z,X}) - \ind(t_{X,Y,Z}) \\
  &= (d - \tfrac{5}{2}) + (d - \tfrac{5}{2}) + (d - \tfrac{5}{2}) + (d - \tfrac{5}{2}) \\
  &\quad - (d - \tfrac{3}{2}) - (d - \tfrac{3}{2}) - (d - \tfrac{3}{2}) - (d - \tfrac{3}{2}) \\
  &= -4.
\end{align*}
 \end{proof}
 
We now complete a proof of Theorem~\ref{thmSurface}.  
\begin{proof}
First, $\St$ is preserved under any jump \textup{(E)} and any jump \textup{(H)}, since no triple points are involved in these moves.  
Second, by Proposition~\ref{Prop:Tmove}, $\St$ changes by $2\ind(t)$ under the jump \textup{(T)}, where this $t$ refers to one of the two triple points that appear or disappear in a single (\textup{T}).    
Third, Proposition~\ref{Prop:Qmove} implies that each jump \textup{(Q)} changes the value of $\St$ as in the Table~\ref{table:ValueOnly}:

\begin{table}[h!]
\caption{Changes in $\St$ under each case of quadruple point jumps.}\label{table:ValueOnly}
\begin{tabular}{|c|c|c|c|}\hline
$Q^j$ strata&Q$_i$-moves&Opposite Q$_i$-moves & Changes in  $\St$ \\ \hline
Negative $Q^4$ & & Opposite $\textup{Q}_0$ & $+4$ \\ \hline
Negative $Q^3$ & $\textup{Q}_3$ & Opposite $\textup{Q}_1$ & $+2$ \\ \hline
 $Q^2$ & $\textup{Q}_2$ & $\textup{Q}_2$ & $0$ \\ \hline
Positive $Q^3$ &$\textup{Q}_1$ & Opposite $\textup{Q}_3$ & $-2$ \\ \hline
Positive $Q^4$ & $\textup{Q}_0$& & $-4$ \\ \hline
\end{tabular}
\end{table}

\noindent This table completes the proof of the claim.     
\end{proof}
\section{Applications}\label{sec:Appl}
In this section, we apply the invariant $\St$ (Theorem~\ref{thmSurface}) to  sphere eversions ($\Sigma=S^2$).  
\begin{corollary}[$S^2$ version of Theorem~\ref{thmSurface}]\label{thmSphere}
Let $S$ be a generic sphere immersion $S^2 \looparrowright \mathbb{R}^3$; $T(S)$ the set of triple points of $S$; and $\ind(t)$ the Alexander numbering of the triple point $t$,  defined as the average of the indices over the eight regions adjacent to $t$. 
Then the value
\[\St(S)= \sum_{t \in T(S)} \ind(t)
\]
behaves under jump singularities as in Table~\ref{table:ValueOnly}.  
For \textup{(T)}, the change is an integer $2n$ that is $2 \ind(t)$ of the   triple points   created during a jump \textup{(T)}.
Moreover, $\St$ remains unchanged under the other types of singularity jumps:  \textup{(E)} and \textup{(H)}.
\end{corollary}

Let $|T(S)|$ be the cardinality of the set $T(S)$.  
We detect (\textup{T}), as well as positive or negative $Q^3$ and positive or negative $Q^4$, that occur during a smooth sphere eversion process, by applying the function $\frac{1}{2}\St$ as follows: 
\begin{enumerate}
\item If $|T(S)|$ changes, then the change corresponds to the appearance or disappearance of a pair of triple points created or eliminated by a jump (\textup{T}), both having the same value $\ind(t)=n$.  
\item If $|T(S)|$ does not change, then $\St$ changes depending on the cases of Goryunov's $Q^i$.  In particular, $Q^3$ and $Q^4$ affect  $\frac{1}{2}\St$, while $Q^2$ does not affect $\frac{1}{2}\St$, as  shown in Table~\ref{table:halfSt}: 
\begin{table}[h!]
\caption{Changes in $\frac{1}{2}\St$ under each case of quadruple point jumps.}\label{table:halfSt}
\begin{tabular}{|c|c|}\hline
Cases  & Changes in $\frac{1}{2}\St$ \\ \hline
Negative $Q^4$  & $+2$ \\ \hline
Negative $Q^3$  & $+1$ \\ \hline
 $Q^2$  & $0$ \\ \hline
Positive $Q^3$  & $-1$ \\ \hline
Positive $Q^4$  & $-2$ \\ \hline
\end{tabular}
\end{table}
\end{enumerate}
As a corollary, we have Proposition~\ref{propAppl}.  
\begin{prop}\label{propAppl}
Let $S$ be an immersion of the $2$-sphere that occurs during a sphere eversion.  Then:
\begin{itemize}
    \item $\#\{$occurrences of \textup{(T)} before reaching $S \} \ge \frac{1}{2}|T(S)|$.  
    \item Let $\mathcal{T}$ be the set of triple points associated with jumps \textup{(T)} that occurred before reaching $S$. Then the difference  
    \[
    \frac{1}{2}\St(S) - \sum_{t \in \mathcal{T}} \varepsilon(t)\ind(t)
    \]
    represents the contribution to $\St$ arising solely from jumps \textup{(Q)}, where the sign $\varepsilon(t)$ is determined by the direction of the jump \textup{(T)}: it is $+1$ if the triple point $t$ is created, and $-1$ if it is eliminated.
    \item Since $\St$   changes by at most $2$ per jump \textup{(Q)}, the value
    \[
     \# \{ \textit{occurrences of } \textup{(Q)} \textit{ before reaching S} \} \ge
    \frac{1}{2}\left|\frac{1}{2}\St(S) - \sum_{t \in \mathcal{T}} \varepsilon(t) \ind(t)\right|. 
    \]

\end{itemize}
\end{prop}

\section*{Acknowledgement}
The authors would like thank Professor Keiichi Sakai for his comments.  

\bibliographystyle{plain}
\bibliography{hStRef}
\end{document}